\documentclass{amsart}
\usepackage{verbatim,amssymb,xcolor,rotating,ulem}

\newtheorem{thm}{Theorem}[section]
\newtheorem{prop}[thm]{Proposition}
\newtheorem{cor}[thm]{Corollary}
\newtheorem{lem}[thm]{Lemma}
\newtheorem{rk}[thm]{Remark}
\newtheorem{conject}[thm]{Conjecture}
\newtheorem{ex}[thm]{Example}

\begin{document}

\title[Big groups of automorphisms]{On the biggest purely non-free  conformal actions on compact Riemann surfaces
 and  their asymptotic properties}

\author{C. Bagi\'nski, G. Gromadzki, R. A.  Hidalgo}

\thanks{C. Bagi\'nski, G. Gromadzki  supported by NCN 2015/17/B/ST1/03235.  R. A. Hidalgo supported by FONDECYT 1230001}



\maketitle

\begin{abstract}
A continuous action of a finite group  $G$ on a closed orientable surface $X$ is said to
be gpnf (Gilman purely non-free) if  every element of $G$ has a fixed point on $X$.
We prove that the biggest order {$\mu(g)$}, of a gpnf-action on a surface of even genus
$g \geq 2$, is bounded below by $8g$ and that this bound is sharp for infinitely many even $g$ as well.
This provides, for even genera,  a gpnf-action analog of the celebrated  Accola-Maclachlan bound
$8g+8$ for arbitrary finite continuous actions.
We also describe the asymptotic behavior of $\mu$. We define $\mathcal{M}$ as the set of values
of the form $$\widetilde{\mu}(g)=\frac{\mu(g)}{g+1},$$ and its subsets
 $\mathcal{M}_+$ and $\mathcal{M}_-$ corresponding to even and odd genera $g$. We show that
the set $\mathcal{M}_+^d$, of accumulation points of $\mathcal{M}_+$, consists of a single number $8$.
If $g$ is odd, then we prove that $4g \leq \mu(g)<8g$. We conjecture that this lower bound is sharp for infinitely many odd $g$.
Finally, we  prove that this conjecture implies that $4$ is the only element of
$\mathcal{M}_-^d$, leading to $\mathcal{M}^d=\{4,8\}.$
\end{abstract}

\section{Introduction}
Let $X= X_g$ be a closed orientable surface of genus $g \geq 2$ and let  $\nu(g)$ be the biggest
order of a finite group $G$ of orientation-preserving homeomorphisms of $X$. The finiteness of
$G$ guarantees that we can find a Riemann surface structure on $X$ for which $G$ is
a group of its conformal automorphisms. This fact allows us to employ algebraic machinery
based on the Riemann uniformization theorem and the well-developed combinatorial theory of
Fuchsian groups.

\medskip
The classical Hurwitz upper bound \cite{Hur} asserts that 
$\nu (g) \leq 84(g-1)$, while the Accola-Maclachlan \cite{A,M} lower bound asserts $\nu (g)\geq 8(g+1)$.
Both  bounds are known to be attained and not attained for infinitely many $g$.
Numerous articles, devoted to the study of large finite groups of homeomorphisms,
can be seen just as a study of $\nu(g)$, but  precise values are only known for specific
cases and series of $g$. However, determining all values of the function $\nu$ seems to be
an unfeasible task.
Special attention has been  given to finding the genera $g$ for which the Hurwitz
bound is attained, i.e., for which $\nu  (g)=84(g-1)$. Less attention has been paid to
another interesting problem of describing those $g$ for which the Accola-Maclachlan bound is exact,
and obtaining information about the values of $\nu$ for remaining $g$.
In recent papers \cite{BG,G-add} we have completely described the asymptotic properties
of the function  $\nu$.

The function
$\nu = \nu(g)$ can be generalized
to  the function $\nu_{\mathcal G}$
by restricting  the spectrum of groups acting on surfaces to
some special classes
${\mathcal G}$
such as, for example, soluble or nilpotent groups (see e.g. \cite{BEGG,GM,Zom1, Zom2}).
 Alternatively, it can be generalized to  $\nu_{\mathcal G}^{{\mathcal S},{\mathcal P}}`$
assuming that we allow only  actions of groups from ${\mathcal G}$ on surfaces of certain
topological type ${\mathcal S}$
such as, for example, bordered or nonorientable (see e.g. \cite{BEGG, M12,M4})
or actions of groups having certain particular  properties ${\mathcal P}$,
such as, for example, hyperelliptic or $p$-gonal  (see e.g. \cite{BEGG,BCG}).

In this paper, we study $\nu^{\rm gpnf}$ for purely non-free actions of arbitrary groups.
For notational simplicity, we shall  denote this function by $\mu$. Such actions were
introduced by J. Gilman in  \cite{Gil} in her study of adapted bases for the actions induced
on the first integral homology. They are defined as  those in which all elements have fixed
points;  we shall refer to them as to {\it gpnf-actions}. The investigation of gpnf-actions
was initiated in \cite{BGH}, where  we showed that such an action on a closed orientable
surface exists for an arbitrary finite group $G$, and we have solved the minimal genus problem
for abelian groups. The present paper can be seen as a step towards the solution of the
maximal order problem for such actions.

In Theorem \ref{A-M for even gpnf}, we prove, for $g \geq 2$ even, that $\mu(g) \geq 8g$,  and
that this lower bound is sharp for infinitely many even $g$. This implies that there are infinitely
many $g$  for which there are no gpnf-actions of order greater than $8g$.
Additionally, in Theorem \ref{form of mu}, we  obtain  that for the remainder even values of $g$ it holds that
\begin{equation}\label{eq1}
\mu(g) =
\dfrac{8n}{n-4}
(g-1),
\end{equation}
for some positive integers $n>4$. Furthermore, Zelmanov's solution to the Restricted Burnside Problem,
\cite{Z1,Z2}, allows us to conclude  that given $n$, any value  from (\ref{eq1}) may be taken for
at most finitely many $g$, which precisely describes asymptotic behavior of {$\mu$}.  Specifically,
in Theorem \ref{derivative},  we show  that the  set $\mathcal{M}_+^d$
of the accumulation points of the set $\mathcal{M}_{+}$  of values of the  function
 $$
 \widetilde{\mu}(g) = \frac{\mu (g)}{g +1}
 $$
for even genera $g$,  consists of the single number $8$ and so, in particular,
$\{8\} \,\subseteq \,  \mathcal{M}^d$.

The situation for odd $g \geq 3$ is more complex and our results are not as comprehensive (complete instead of comprehensive?). In Subsection 4.2, we
show that $4g \leq \mu(g) < 8g$, for arbitrary $g$.
We conjecture that $\mu (g) =4g$ for infinitely many odd $g$. We show that, if this conjecture holds, then, apart from $8$, the number $4$ is the only other element of the set
$\mathcal{M}^d$ which would give
  $  \mathcal{M}^d =\{4,8\}$.
This demonstrates both, similarities and marked contrast, with
the corresponding result for arbitrary actions
for which we have shown in \cite{BG,G-add} that $(\mathcal{M}^d)^d =\{8,12\}$, within the 
notations adopted to the present paper.

Finally, we note that our results have other interesting features.  Firstly,
thanks to Belyi's Theorem \cite{Bel}, all surfaces realizing the values of
$\mu= \mu^{\rm gpnf}$
can be defined over the field of algebraic numbers. On the other hand, it easily follows
from an old theorem of Singerman \cite{Singerman} that such surfaces can be defined over the reals.
This asserts, due to result of K\"ock and
Singerman \cite{KS}, that such surfaces enjoy  both properties
simultaneously, i.e., they can be defined over the real algebraic numbers. This is a remarkable
fact in the modern study of Grothendieck dessins d'enfants and the inverse Galois problem
\cite{GG}. All of these, combined with the specificity of gpnf-actions, should make their
investigation, from a geometric point of view,  important and fruitful.
In Examples \ref{model 8g} and \ref{model 4g}, we provide geometric models of gpnf-actions of
orders  $8g$ and $4g$ for even and odd genera $g$, respectively.

\section{Preliminaries}
The Nielsen geometrization theorem asserts that a finite group $G$ of homeomorphisms of a closed
surface $S=S_g$ of genus $g \geq 2$ can be seen as a group of automorphisms of a Riemann surface
with respect to some conformal structure on $S$.
On the other hand, the Riemann uniformization theorem allows us to represent  a compact  Riemann surface
$S$  as the  orbit space ${\mathcal H}/\Gamma$ with the conformal structure inherited from the
hyperbolic plane ${\mathcal H}$, where $\Gamma\cong \pi_1(S)$ is a Fuchsian group of signature
$(g;-)$ --- a torsion-free discrete  and cocompact  subgroup of the group of isometries of
${\mathcal H}$. Furthermore, since $\Gamma$ is torsion-free and  due to elementary
theory of coverings, $G$ can be realized as $\Delta/\Gamma$, for some other Fuchsian group $\Delta$.
Precisely, if $\pi: {\mathcal H} \to X={\mathcal H}/\Gamma$ and $\theta : \Delta \to \Delta/\Gamma= G$
denote the canonical projections, then for $t\in G$ and  $x \in X$,
\begin{equation}\label{action}
t(x)=\pi(\delta(y)),
\end{equation}
where $\delta\in \Delta$ and $y \in \mathcal{H}$ are arbitrary elements satisfying
$t=\theta(\delta)$ and $x=\pi(y)$.

A Fuchsian group  $\Delta$ has a presentation
\begin{equation}\label{presentation}
\langle \, a_1, b_1, \ldots , a_h,b_h, x_1, \ldots, x_r \,: \, x_1^{m_1}, \ldots , x_r^{m_r},  x_1 \ldots x_r [a_1, b_1] \ldots  [a_h,b_h] \,  \rangle,
\end{equation}
which is concisely expressed in its signature
\begin{equation}\label{signature}
(h;m_1, \ldots, m_r).
\end{equation}

Within this terminology, a surface group  $\Gamma$  has signature $(g;-)$.

The hyperbolic area of an arbitrary fundamental region of the Fuchsian group $\Delta$ is 
\begin{equation}\label{area}
m(\Delta) = 2 \pi \left(2h-2+ \sum_{i=1}^r \left(1 - \frac{1}{m_i}\right)\right),
\end{equation}
and for a subgroup $\Delta' $ of $\Delta$ of finite index, we have the following Riemann-Hurwitz  formula
\begin{equation}\label{RH}
 [\Delta: \Delta'] = \dfrac{m (\Delta')}{m (\Delta)}.
\end{equation}
The signature $(h;m_1,\dots,m_r)$ in the case $h=0$ will be written shorter as
$(m_1,\dots,m_r)$.
A useful property of $\Delta$ is that any of its nontrivial  elements of finite order is
conjugated in $\Delta$ to a power of some $x_i$. Moreover,  elements of finite order are the
only ones that have a (unique) fixed point in $\mathcal H$.
Therefore, we can easily derive the following two lemmas.

\begin{lem}\label{smooth}
Let $G$ be a finite group and  let $\Delta$ be a Fuchsian group with signature
$(\ref{signature})$. Then an epimorphism  $\theta : \Delta \to G$  has a torsion-free kernel
if and only if it preserves the orders of $x_i$. In such a case, $\ker \theta$  is isomorphic
to the fundamental group of a closed orientable surface whose genus is given directly
by $(\ref{area})$ and  $(\ref{RH})$.
\end{lem}

An epimorphism described in the above lemma is said to be a {\it smooth} or a {\it surface-kernel}
epimorphism, and a corresponding finite group $G$ will be called a $(h;m_1, \ldots , m_r)$-{\it group}
or a $(m_1, \ldots , m_r)$-{\it group} if $h=0$.

\begin{lem}\label{elliptic}
An element  $t \in G$   has a fixed point in $X$ if and only if it is conjugated in $G$ to a power
of $\theta(x_i)$ for some $i$.
\end{lem}
\begin{proof}
Indeed, by (\ref{action}),
$t(x)=x$ if and only if  $\pi (\delta(y))=\pi(y)$, so $\gamma (\delta(y))=y$ for
some $\gamma \in \Gamma$. This implies that $\gamma\delta$ has a fixed point, so
$\gamma\delta = \delta_ix_i^{n_i}\delta_i^{-1}$ for some $\delta_i\in \Delta$ and $n_i<m_i$.
Since $\theta(\gamma)=1$ and $\theta(\delta)=t$, the assertion follows.
\end{proof}

Recall also, after the Introduction, that a group $G$ of self-homeomorphisms of a surface $X$
is said to be a gpnf-action if each $g \in G$ has a fixed point. Lemma  \ref{elliptic}
means, in particular, that

\begin{cor}\label{qp-gpnf}
An action of $G$ defined by a smooth epimorphism $\theta: \Delta \to G$, where $\Delta$ is a
Fuchsian group with signature $(\ref{signature})$, has the gpnf property if and only if
each element
of $G$ is conjugated in $G$ to a power of some $\theta (x_i)$.
\end{cor}

Let us mention also a  method of relating presentations of a Fuchsian group  $\Delta$ and its
subgroup $\Delta'$ of finite index in terms of the regular representation of $\Delta'$ on the
set of cosets $\Delta/\Delta'$  proposed by Singerman in \cite{Sin}.
Its particular case, which we shall use, is given in the following lemma.

\begin{lem}\label{sing}
Let $\Delta$ be a Fuchsian group with signature $(\ref{signature})$  and
let $\theta: \Delta \to G$ be an epimorphism onto a finite group $G$ of order $N$.
Then each period $m_i$ of $\Delta$   give rise to  $N/n_i$ periods $m_i/n_i$ of the
group $\Delta' = \ker \theta$, where $n_i$ is the order of $\theta(x_i)$. The orbit
genus of $\Delta' $ is determined in such case by the Riemann-Hurwitz formula.
\end{lem}

Finally,  to avoid misleading conclusions, we also make the following

\begin{rk}\label{g-p vs action}\rm
Throughout the  paper,  apart from the concept of  gpnf-action,  we shall use  also the  technical
concept of a  {\it gpnf-group}.   Namely, we say that a group
$G$ is a  {\it  gpnf  $(m_1, \ldots, m_r)$-group} or a {\it gpnf-group of type $(m_1, \ldots , m_r)$}
if it is generated by elements  $a_1, \ldots , a_r$  of orders $m_1, \ldots, m_r$,
respectively, such that $a_1 \ldots  a_r=1$  and each element  of $G$ is conjugated in $G$ to a
power of $a_i$.
These two concepts coincide if and only if
$$
-2+ r - (1/m_1 + \ldots  + 1/m_r)
$$
is positive. Observe however that in the other case, the concept of gpnf
$(m_1, \ldots, m_r)$-group or gpnf-group of type $(m_1, \ldots, m_r)$
still works; the alternating  group $A_5$ is a gpnf-group of
type $(2,3,5)$ but there are no  gpnf $(2,3,5)$-action  on a surface of genus $g\geq 2$;
the only such action is the one topologically equivalent to the action on the
boundary of a solid dodecahedron which is a surface of genus $1$.
\end{rk}

\section{There are no gpnf  $(2,3,n)$-actions}
We begin this section with  a brief  introduction to the theory of so-called
{\it prime graphs}  of a finite group $G$, called also {\it the Gruenberg-Kegel graphs}
and traditionally denoted by ${\rm GK}(G)$.
Given a positive integer $n$, let $\pi(n)$ denotes the set of all primes dividing $n$.
The prime graph ${\rm GK}(G)$ of a finite group $G$ is a graph whose set of vertices is equal
to $\pi(|G|)$ and two primes $r,s\in\pi(|G|)$ are adjacent, denoted by $r \sim s$,
if there exists an element of order $rs$.
By a {\it clique} we mean a graph in which every pair of vertices is adjacent.
Following \cite{VV1}, we denote by $t(G)$  the independence number of ${\rm GK}(G)$,
which is the maximum cardinality of an anticlique of ${\rm GK}(G)$,  understood
as a subgraph in which no two vertices are adjacent. In \cite{VV1}, the number $t(G)$
is calculated for simple groups $G$, and examples of anticliques $\rho(G)$ of cardinality
 $t(G)$ are given. It is shown, in particular, that  $t(G)\geq 2$
for any simple group $G$. For our purpose we need to identify all finite simple groups for
which $t(G)=2$.  We list them in Table \ref{t(G)=2}, where
$q$ denotes a power of a prime $p$, which stands for
the order of a field over which a simple finite group of Lie type is considered.
For an integer $a$, the value  $a_r$ for  a prime $r$ denotes the maximal power of $r$
dividing $a$. Finally, by  $r_m$ it is meant  a  prime for which
$$
q^m\equiv 1\pmod{r_m}
$$
and $m$ is the smallest integer with this property.
\begin{table}
\begin{tabular}{cccc}
\hline
&&\\[-2mm]
\;\;\;\;${\rm No} \;\;\;\; $ &{\rm Group} & \;\;\;\;$ \rho(G)$\;\;\;\;\\ [4pt]
\hline
&&\\[-7.1pt]
1 & $J_2$ & $\{5,7\}$ \\[2.1pt]
2 & \;\;\;\;\;\;\;\;\;\;\;\;\;\;\;\;\; $A_2(q),\ (q-1)_3\neq 3,\ q+1=2^k$ \;\;\;\;\;\;\;\;\;\;\;\;\;\;\;\;\; & $\{p,r_3\}$   \\[2.1pt]
3 & $ ^2A_2(q),\ (q+1)_3\neq 3,\ q-1=2^k$ & $\{p, r_6\}$ \\[2.1pt]
4 & $A_3(2)$ & $\{2,5\}$  \\[2.1pt]
5 & $B_2(q),\ q>2$ & $\{p, r_4\}$ \\[2.1pt]
6 & $ C_3(2) $ & $ \{5,7\} $ \\[2.1pt]
7 & $ D_4(2) $ & $ \{5,7\} $ \\[2.1pt]
8 & $ ^3D_4(2) $ & $ \{2,13\} $\\[4.1pt]
\hline
\end{tabular}\\[2mm]
\caption{Simple groups with $t(G)=2$. \label{t(G)=2}}
\end{table}

\medskip
We shall also need the following well-known results. The first one is a characterization
of finite simple groups with elementary abelian Sylow  $2$-subgroups by Walter \cite{Walter}.

\begin{lem}\label{elementary abelian_2_groups}
If $G$ is a finite simple group with elementary abelian Sylow $2$-subgroups, then $G$ is
isomorphic to one of the groups
\begin{itemize}
	\item[{\rm (a)}] $A_1(q)={\rm PSL}(2,q)$, $q = 2^n$, $n \geq 2$, or $q \equiv 3,5 \pmod{8}$,\vspace{0.4pt}
	\item[{\rm (b)}] Janko's first group $J_1$,\vspace{0.4pt}
	\item[{\rm (c)}] A  Ree group $^2G_2(3^{2n+1})$ for some $n$.\hfill $\qed$
\end{itemize}
\end{lem}

The second result is a characterization of finite simple groups without elements
of order six by Fletcher, Stellmacher, Stewart \cite{FSS}.

\begin{lem}\label{elements of order 6}
Let $G$ be a finite simple group that has no elements of order six. Then $G$ is isomorphic to one of the following groups:
\begin{enumerate}
\item[{\rm (a)}] $A_1(q)={\rm PSL}(2,q)$, $q \neq \pm 1 \pmod{12}$,\vspace{0.4pt}
\item[{\rm (b)}] $A_2(2^n)=PSL_3(2^n)$, $n\geq 2$, $n\not\equiv 0 \pmod{6}$,\vspace{0.4pt}
\item[{\rm (c)}] $^2A_2(2^n)=PSU(3, 2^n)$, $n\geq 2$, $n\not\equiv 3, 5 \pmod{6}$,\vspace{0.4pt}
\item[{\rm (d)}] $Sz(2^{2n+1})$, $n \geq 1$. \hfill $\qed$
\end{enumerate}
\end{lem}


\begin{lem}\label{cliques}
	If $G$ is a $(p,q,n)$-group with a gpnf-action, where $p,q$ are primes, then all
    connected components of ${\rm GK}(G)$ are cliques. In particular, there are at most $3$
    such components and then $t(G)\leqslant 3$.
\end{lem}
\begin{proof}
Observe that  $\pi (n)$ forms a clique in  ${\rm GK}(G)$. Let $a$ be an element of $G$ of
order $n$  and $r$  an arbitrary element of $\pi (n)$. If $pq$ divides  $n$, then  $\pi(|G|)$
forms  a clique  since  $a^{n/p}$, $a^{n/q}$ and $a^{n/r}$ are pairwise commuting elements of order
$p$, $q$, $r$ respectively.  If, on the other hand, $p$ divides $n$ and $q$ does not,
then $a^{n/p}$
and $a^{n/r}$ are such elements of order $p$ and $r$, respectively.  Additionally $q$ is not  adjacent
to $r$ since otherwise we would have an element of order $qr$ in $G$,  and therefore,  due to the gpnf
property  $qr$ would divide $n$  and so ${\rm GK}(G)$ would be the union of two cliques  $\{q\}$ and
$\pi(n)$. Finally,  for $p$ and $q$ not dividing $n$, ${\rm GK}(G)$ would be the  union of three
cliques spanned by $\{p\}$, $\{q\}$ and $\pi (n)$.
\end{proof}

\begin{lem}\label{simple gpnf}
The only finite simple gpnf $(2,3,n)$-group is the alternating group $A_5$. Furthermore there
are no gpnf-actions of simple groups with signatures $(n,\stackrel{r}{\ldots}\,,n)$,
 $(2,\ldots ,2,n)$ and $(3, \ldots,3,n)$.
\end{lem}
\begin{proof}
 Let $G$ be a gpnf $(2,3,n)$-group generated by elements $x$ and  $y$ of orders $2$ and $3$
respectively, whose product  $z$ has order $n$. Then, by the definition of a $gpnf$-group
\begin{equation}
G=\bigcup_{g\in G}(\langle x\rangle^g\cup\langle y\rangle^g\cup\langle z\rangle^g).
\end{equation}
Therefore $G$ has no elements of order $6$, since such an element would be conjugated to a power of $z$.
Thus $6$ would divide $n$, contradicting $t(G) \geq 2$, as ${\rm GK}(G)$ would be a clique.
If  $n$ is even then  ${\rm GK}(G)$ is  the  union of two cliques spanned by $3$ and $\pi (n)$,
and thus $G$ would be isomorphic to one of the groups listed in Table \ref{t(G)=2}. However, this is
impossible since all these groups must contain elements of order $6$ by Lemma \ref{elements of order 6}.
Finally, if $n$ is odd, then a $2$-Sylow subgroup of $G$  is elementary abelian and then it belongs to the
groups given in Lemma \ref{elementary abelian_2_groups}. But $t(G) = 5$ for the Ree group and
$t(G)=4$  for the first Janko group. Eventually,  for ${\rm PSL}(2,q)= A_1(q)$,   $\{p,r_1,r_2\}$  forms
a maximal coclique by \cite{VV1}, while according to \cite{VV2}, this coclique is unique
among maximal ones. On the other hand, it is well-known that ${\rm PSL}(2,q)$ for  $q=p^k$ has,
up to conjugation, three maximal cyclic subgroups of orders $(q+1)/2,  (q-1)/2$ and $p$ (see for
example \cite{Huppert}), which,  due to the gpnf property, implies that
$\{2,3,n\}=\{p, (q-1)/2, (q+1)/2\}$ which is possible only for $n=q=p=5$ resulting in $G=A_5$.

For  a gpnf $(n, \ldots, n)$-group, ${\rm GK}(G)$ would be a clique spanned by $\pi (n)$,
which is not the case.  For $(2,\stackrel{r}{\ldots}\, ,2,n)$ with  odd $n$, a Sylow $2$-subgroup
of $G$  is elementary abelian which was ruled out in the previous part of the proof.
For even $n$, ${\rm GK}(G)$ would be a clique again, spanned by $\pi (n)$  once more.

Finally let $G$ be a gpnf $(3,\ldots ,3,n)$-group and let $z$ be a fixed element of order $n$.
Any nontrivial element of order not equal to $3$ is conjugate to a power of $z$, and thus its order
divides $n$. Hence, all primes not equal to $3$ dividing $|G|$ are adjacent in ${\rm GK}(G)$.
If $3$ divides $n$ as well, then ${\rm GK}(G)$ is a clique spanned by $\pi(n)$, which is not
possible. So assume that $3$ does not divide $n$. Then ${\rm GK}(G)$ is a union of two cliques
spanned by $3$ and $\pi (n)$, and therefore $G$ is one of the groups listed in Table \ref{t(G)=2}.

Let $\rho(G)=\{r,s\}$, where $r,s\neq 3$. Then elements of orders $r$ and $s$ are
conjugate to a power of $z$, which means that if $r$ and $s$ are adjacent in ${\rm GK}(G)$,
it leads to a contradiction since they form an anticlique. Thus $G$ cannot appear in any of
items 1, 4-8 of Table \ref{t(G)=2} because they contain elements of order six by Lemma
\ref{elements of order 6}. Next, suppose that $G$ appears in item 2 in the table.
If $q=p^m$ with $m>1$, then $p^m=2^k-1$, which is not possible by Mihailescu solution
of the Catalan Conjecture \cite{Mih}. Thus $q=p$ is an odd prime, and then $G$ contains elements
 of order six by Lemma \ref{elements of order 6}. This excludes such groups.
Similar arguments verify that $G$ appears in item 3. Indeed, here $q$ cannot be a proper power of $p$
 since otherwise $p^m = 2^k+1$, and again by Mihailescu's Theorem, either $q=9$, $k=3$, or $q=p$ is
 an odd prime. In both cases, $G$ contains elements of order $6$ by Lemma \ref{elements of order 6},
 which finishes the whole proof.
\end{proof}

\begin{cor}\label{no 23n simple gpnf}
There are no simple  gpnf $(2,3,n)$-actions on a compact Riemann surface of genus $g \geq 2$
\end{cor}
\begin{proof}
Indeed,  for the exceptional case given in  Lemma
\ref{simple gpnf} we have $n=5$ which leads to the spherical  $(g=0)$ and not to a hyperbolic case.
\end{proof}

\begin{lem}\label{one cyclic}
Let $H$ be a normal subgroup of a finite group $G$ such that
\begin{equation} \label{one cyclic formula}
H= \bigcup_{g \in G}\; K^g
\end{equation}
 for  a cyclic subgroup $K$ of $H$.
  Then $H$ is solvable.
\end{lem}
\begin{proof}
We shall prove the lemma by induction on the order of $H$. Let $K= \langle \,  h \,  \rangle$
for $h \in H$ and assume first that $H$ contains a proper subgroup $N\neq 1 $,   normal in
 $G$. Then for
 $\widetilde G=G/N, \widetilde H= H/N, \tilde h=hN$, $K'= \langle \, h^k \,  \rangle$,
 where $k$ is the order of $\tilde h$, we have
$$
\widetilde H= \bigcup_{\tilde g\in \widetilde G} \;\langle \, \, \tilde  h  \,\rangle^{\tilde g} \;\; {\rm and } \;\;  N=\bigcup_{g\in G} \; K'^g.
$$
Hence,   by the inductive assumption, $N$ and $H/N$ are solvable and so  $H$ is solvable also.
Assume that $H=N$ which means that $H$ is a minimal normal subgroup of $G$. If $H$ is non-abelian,
then $H=A\times \stackrel{r}{\ldots} \times A$  for  a non-abelian finite simple group $A$.
But for $r=1$, the group $H=A$ has at least two nonisomorphic maximal cyclic subgroups
since otherwise the prime graph ${\rm GK}(A)$  would be a clique. Therefore, $H$ cannot be
presented as a union of isomorphic cyclic groups.
On the other hand, if $x\in A$ generates a maximal cyclic subgroup of $A$, then
$(x, \stackrel{r}{\ldots}, x)$  generates a maximal cyclic subgroup of
$H=A\times \stackrel{r}{\ldots} \times A$ and therefore also for $r>1$,  $H$ has
at least two nonisomorphic maximal  cyclic subgroups which is not our case in which $H$
is covered by a family of isomorphic cyclic groups, and so our proof is complete.
\end{proof}

Before proceeding to the next lemma, let's recall that a group is said to be {\it perfect} if
its abelianization is trivial.

\begin{lem}\label{perfect gpnf}
There are no perfect gpnf $(p,\ldots ,p,n)$-group, for $p = 2,3$,
and the only perfect gpnf $(2,3,n)$-group  is $A_n$ with $n=5$.
\end{lem}
\begin{proof}
 Let $G$ be a perfect gpnf $(p,\stackrel{r}{\ldots}\, ,p,n)$-group  of the smallest possible
 order.  Then, by the second part of Lemma  \ref{simple gpnf}, $G$ is non-simple. Let $H$ be
 a nontrivial and proper normal subgroup of $G$. Then $G/H$ is a perfect
 $(p,\stackrel{r'}{\ldots}\, ,p,n')$-group, which is impossible  by the minimality of $G$.

Then let $G$ be a perfect gpnf $(2,3,n)$-action for some $n\geq 7$ of minimal possible order,
and let $a$ and $b$ be generators of orders $2$ and $3$ whose product $c$ has order $n$, and
any element of $G$ is conjugated to $a$ or to a power of $b$ or $c$.
By Lemma \ref{simple gpnf}, $G$ is not simple. Let $H$ be a nontrivial normal  subgroup of $G$
of minimal order.
Then   $\tilde c \in G/H$ is an element of order $k \leq 6$,  while $\tilde a$ and
$\tilde b$  have orders $2$ and $3$. For $k=1,2,3,4$  $G/H$  would be respectively
trivial (which is not possible as $H$ is a proper subgroup of $G$), the alternating groups
$A_3, A_4$ or the symmetric group $S_4$ and none of them is perfect.
For  $n=6$,  $G/H$ is an extension of the abelian group  $Z_m^2$ by $Z_6$ by the
well-known result of Miller  \cite{Mil}  (see also \cite{S}). So $k=5$, and therefore $G/H=A_5$.
{The induced action} of $H$ is purely non free while by \cite{Sin} (see also more
available \cite[Theorem 2.2.4]{BEGG}), this action   of $H$  is of type
$(0;m, \stackrel{12}{\ldots}\,,m)$-group  for $m= n/5$
which precisely means that $H$ is generated by twelve elements of order $m$  with the
trivial product, and any element of $H$ is conjugated in $H$ to a power of  one of these
twelve elements, which means that $H$ is a gpnf $(0;m, \stackrel{12}{\ldots}\,,m)$-group.
Any  element of $G$ of order $n$ is conjugated in $G$ to a power of $c$ and therefore any
of these twelve elements is conjugated in $G$ to a power of $c^5$, and so  by Lemma
\ref{one cyclic},
$H$ is solvable. Observe that  $H$ cannot contain a proper nontrivial subgroup $N$ being
normal in $G$, since in such a case $G/N$ would  be a  gpnf  $(2,3,n')$-group  for some
$n' \geq 7$, contrary to the minimality of $G$.
So $H$ is an elementary abelian $p$-group, say of order $p^s$. Now, since $H$ is abelian,
the transitive action of $G$  on the set of cyclic subgroups of $H$ induces a transitive
action of $A_5$ on this set. Furthermore $\tilde c \in G/H =A_5$  being of order  $5$
belongs to  the  stabilizer of the point represented by the subgroup
$\langle \, c^5 \, \rangle$. Therefore it has index $1,6$ or $12$  since $A_5$ has no
subgroups of index   $\leq 5$.
Therefore, the unique orbit of this action has length $l=12, l=6 $  or $l=1$. The last would
imply that $H\simeq Z_p$ and therefore $G$ would be a $(2,3,5p) $ of order $60p$.
 This is impossible since the action of $A_5$ on $H=\langle c^5\rangle$ is trivial due to
 the simplicity of $A_5$ and then either $ac^5$ has order $2p$ (when $p\neq 2$) or $bc^5$
 has order $3p$ (when $p\neq 3$). There are no elements of such order in $G$.

Next, the number of cyclic subgroups of $H$ is equal to
$$
1+p+\cdots+p^{s-1}= \frac{p^s-1}{p-1},
$$
which is not equal to $12$ for any prime $p$. This number is equal to $6$ only for $p=5$
and for $s=2$  i.e.  for $H\cong Z_5^2$.  Again, the action of $A_5$ on $Z_5^2$  would
be trivial since $A_5$ is not contained in ${\rm Aut}(Z_5^2)$.
\end{proof}


\begin{lem}\label{union of cyclics}	
A finite  group is a union of three cyclic subgroups if and only if it is isomorphic to
$T \times Z_k$, where $T$ is either the Klein four-group or the quaternion group of order $8$
and $k$ is odd.
\end{lem}
\begin{proof} It is an easy consequence of a result of Scorza  (cf.  \cite{Zappa}) by
which a group is the set-theoretic union of its three proper subgroups if and only
if it has the Klein four-group as a homomorphic image. Furthermore, Haber and Rosenfeld
\cite{HR}  have shown  that if $G=A \cup B \cup C$, then $H=A \cap B \cap C$ is a normal
subgroup of $G$  and $G/H$ is the Klein four-group which
easily gives the assertion since, in our case,  $H$ is cyclic as the intersection of cyclic groups.
\end{proof}

\begin{thm}\label{arbitrary}
There are no gpnf $(2,3,n)$-actions on surfaces of genus $g \geq 2$.
\end{thm}
\begin{proof}
  Let $G$ be such a group of the smallest possible order and let $x,y$ denote the generators of
  order $2$ and $3$, respectively, whose product $z=xy$ has order $n$.
Then, due to Remark \ref{g-p vs action} and by the last part of the Lemma
\ref{perfect gpnf}, $G$ cannot be perfect.
  So $G_{\rm ab} =G/G'= Z_2, Z_3$ or $Z_6$ and we shall consider these three cases separately.
  Here and in the proof of  Lemma \ref{33n & 344} we shall use some
well-known facts on groups of orders
$16,18,24$ and $27$ (consult for example \cite{TW} or GAP  library of groups of small order
\cite{GAP-Small, GAP}).

Suppose first that $G_{\rm ab}=Z_2$. We shall show first that $G/G^{(3)}=S_4$ and so for
notational simplicity assume for this part that   $G^{(3)}=1$.    Here   $G'$ is a gpnf
$(3,3,m)$-group, where $m=n/2$,
 and by  Lemma \ref{perfect gpnf},   $G'_{\rm ab} = Z_3$ or $G'_{\rm ab} = Z_3^2$.
 But the second case is impossible since
a non-abelian group of order $18$  has an element of order $9$ or it has at least three
conjugacy classes of subgroups of order $3$, contrary to the gpnf assumption on $G$ by
which it has at most two such classes.
So $G/G^{(2)}= D_3$. Therefore, being an abelian gpnf $(m,m,m)$-group,   $G^{(2)}$  is
the  set-theoretical union of at most three cyclic subgroups.
But it is well known that a group cannot be a set-theoretical union of two proper subgroups
and thus we have to consider the two cases only.

First we shall show that $G^{(2)}$ cannot be cyclic. For observe that for cyclic $G^{(2)}$,
we would have  a subgroup $K$ of $G^{(2)}$ normal in $G$
for which $G/K$ would be a gpnf  $(2,3,2p)$-group of order $6p$ for some prime $p$.
But clearly $p\neq 2$ since a $(2,3,4)$-group is equal to $S_4$. Also $p\neq 3$ since a
$(2,3,6)$-group has abelianization $Z_6$ by  aforementioned results \cite{Mil,S}).
Finally also  $p>3$ can not be the case here. Indeed, $K=\langle w \rangle$  for  $w=z^2$ and we
have $w^x= w^{-1}$ since  $w^x=w$ would make $w$ central in $G$ and we would have an element
of order $3p$ in $G$, contrary to the gpnf property of $G$. Next  $w^y=y^\beta$, where  $p$ divides
$\beta^3-1$. So on the one hand  $w^{xy}= w$ since $w=(xy)^2$, while on the other hand,
$w^{xy}= (w^y)^x=w^{-\beta}$, which implies that $p$ simultaneously divides $\beta +1$ and
$\beta^3-1=(\beta-1)\big(\beta(\beta+1)+1\big)$, which is an absurd.
So $G^{(2)}$ is the union of three cyclic groups, and therefore $G^{(2)}=Z_2^2$ by Lemma
\ref{union of cyclics}. 	
Therefore for the initial gpnf $(2,3,n)$-group $G$, $G/G^{(3)}= S_4$ indeed, and now
$G^{(3)}$ is a gpnf   $(m,m,m,m,m,m)$-group  for $m=n/4$. By the minimality of  $G$,
we can assume that   $m=p$  and $G^{(3)}$ is elementary abelian, say  $Z_p^s$. But
now $G^{(3)}$ is the union of at most six cyclic subgroups and so $s=2$ and $p\leq 5$.

Observe that $p\neq 2$ since otherwise $G$ would be a $(2,3,8)$-group of order
$96$, and so  it would act on a surface of genus $3$, contrary to the  classification of Broughton \cite{Br}.
Then also $p\neq 3$ since
otherwise we would have a gpnf  $(2,3,12)$-group of order $2^3 3^3$. But then an arbitrary
Sylow $3$-subgroup $P$  of    $G$  would be a  gpnf $(h;3, \stackrel{t}{\ldots}\,,3)$-group
and so in particular all nontrivial elements of $P$ have order $3$.
But by
(\ref{area}) and  (\ref{RH})
$$
2h-2+ t(1-1/3) =  8(1/12)=2/3
$$
and so $h=1, t=1$  or $h=0,t=4$.
Clearly, $P$ is non-abelian since
otherwise it would be covered by at most four cyclic groups by the gpnf property. In turn,
for  non-abelian $P$, $P_{\rm ab}=P/P'= Z_3^2$. But then in the first case,  $P_{\rm ab}$
would be cyclic, while in the second case,  at least one of the four generators of $P$ would
be in $P'$, and therefore  $G_{\rm ab}$ would be covered by at most three cyclic subgroups.

Finally, suppose, to derive a contradiction, that $p=5$, and $ \pi:G \to S_4$
be the canonical projection. Let $H$  be the preimage
of a  Sylow 2-subgroup $P$ containing  $y^5$. Then, by the Zassenhaus theorem, $H\cong K \rtimes P$,
and  $y\in  H$ since  $y^5 \in P$ and  $y^4 \in K$.
Summing up,  $G$ contains a subgroup $H$ of index $3$, which is a gpnf
 $(h;  m_1, \ldots, m_r)$-group for some $m_i$ and $h$, and at least  one of $m_i$, say $m_1$,
 is  equal to $20$.  But then, again by
 (\ref{area}) and  (\ref{RH})
$$
2h-2+ (1-1/20)+  \sum_{i=2}^r (1-1/m_i)) =  7/20,
$$
while on the other hand, the above equation has no solution in nonnegative integers.

Now suppose that $G_{\rm ab}=Z_3$. Then $G'$ is a gpnf $(2,2,2,m)$-group, where $m=n/3$,
and by Lemma \ref{perfect gpnf}, $G'_{\rm ab} = Z_2^s$ for  where $s=1,2$ or $3$. However,
$s\neq 1$. Furthermore, $s \neq 3$ since the only group of order  $24$  with
abelianization $Z_3$ is ${\rm SL}_2(3)$,  which has just one element of order $2$. Therefore,
$G/G^{(2)}$ is a $(2,3,3)$-group, and so  $G/G^{(2)}=A_4$. Now $G^{(2)}$ is a gpnf $(m,m,m)$-group,
and $G^{(2)}/G^{(3)}$ is a factor group of $Z_m^2$.
But then there exists a characteristic subgroup  $K$ of $G^{(2)}$, so that $G/K$ has order
$12p^s$ for $s=1$ or $s=2$ and it is a gpnf $(2,3,3p)$-group.
Set $\widetilde{G}=G/K$ and $\widetilde{H}=G^{(2)}/K$.
Observe that $p \neq 2$ since a $(2,3,6)$-group has abelianization $Z_6$, by \cite{Mil,S}).
Furthermore $p \neq 3$. Indeed,  for $s=1$, $\widetilde{G}$ would act on a surface of genus
$2$, by  (\ref{area}) and  (\ref{RH}),  contrary to the classification of Broughton
\cite{Br} suplemented by the corrections concerning the $(2^3,4)$-actions
provided by Bogopolski in  (\cite{Bog}, page 15).  In turn for  $s=2$
such a group would act on a surface of genus $4$, contrary to the classifications
of Bogopolski of such actions given in the aforementioned paper  \cite{Bog} (see also the
later paper of Kimura \cite{Kim} with the same result).
Finally, assume,  to derive a contradiction,  that $p\geq 5$.
Then $\widetilde{G}$  acts on the set of subgroups of $\widetilde{H}$ of order $p$ by
conjugation. This action  induces  natural action of $A_4=\widetilde{G}/\widetilde{H}$ on
$\widetilde{H}$  with the same orbits.  Indeed, given   $g,g'\in \widetilde{G}$ for which
$g\widetilde{H}=g' \widetilde{H}$, we have $g'=gh'$ for some   $h' \in \widetilde{H}$, and so,
since  $\widetilde{H}$ is abelian, $ghg^{-1}=g'hg'^{-1}$ for any   $h \in \widetilde{H}$.
 On the other hand,  the orbit of such action has length at most $4$, being the index of
 the stabilizer of $\langle  (ab)^3\rangle$  in $A_4$. Hence
$p+1 \leq 4$, and therefore  $p =2$ or $p =3$, which were just ruled out.

Finally, suppose that $G_{\rm ab}=Z_6$.
Then neither $x$ nor $y$, and no their conjugates belong to $G'$. So since $G$ is a gpnf
$(2,3,n)$-group
$$
G' \subseteq \bigcup_{g\in G} \; \langle xy \rangle ^g.
$$
Therefore, $G'$ is solvable by  Lemma \ref{one cyclic}.
But since $G/G' =Z_6$,  we actually have
$$
G'=\bigcup_{g\in G} \; \langle (xy)^6 \rangle ^g.
$$
Finally, for $x=[a,b], y=[b^2,a]$, we have   $(ab)^6=[x,y]\in G^{(2)}$, and so
$G'=G^{(2)}$, which means that  $G=Z_6$, which is not our case. This completes the proof of
this case and the entire proof of the theorem.
\end{proof}

\section{Accola-Maclachlan lower bound for the order of gpnf-actions}
In this section we prove that $\mu(g)  \geq 8g$ for arbitrary even $g$, and that this bound is
exact for infinitely many $g$. We also discuss the odd genus case. Therefore we divide the section
into two parts according to the parity of $g$. We begin with

\subsection{The even genus case.}

\begin{lem}\label{geq 8g} There exists a gpnf-action of order $8g$ on a surface of arbitrary
even genus $g$.
\end{lem}
\begin{proof}
Consider the group $G$ with the presentation
\begin{equation}\label{atain 8g}
\langle \, x,\,y,\, z\,:\,  x^2,\ y^4, z^{4n}, xyz,\ y^2z^{2n}\, \rangle.
\end{equation}
Then
 $xzx= x(y^{3}x)x=xyy^2=z^{2n-1}$  and so  $\langle \, z \, \rangle$ is a normal subgroup,
 which together with $x$ generates $G$. In particular $|G|=8n$ and $G$ is a $(2,4,4n)$-group.
 Thus, by the Riemann-Hurwitz formula, it acts on a Riemann surface of genus $g=n$ and we shall
show that it is a gpnf-action for even $n$.
For that, we need to show that each element of $G$ is conjugated to a power of $x$, $y$, or $z$.
Now, $[z^{-1},x]=z^{2n-2} $ has order $2n$, taking into account that $n$ is even. Hence
$|G:G'|=4$ and $G/G'\cong Z_2^2$. It is easily seen that the centralizers of $x$ and $y$
are equal to $\langle \, x,\,z^{2n}\,  \rangle$
and $\langle \, y\,  \rangle$ respectively. So the conjugacy classes $x^G$ and $y^G$  of
these elements have size $2n$.
But for arbitrary $a, b\in G$  $aba^{-1}= b[b^{-1},a]\in bG'$. So $x^G=xG', y^G =yG'$ since
these cosets also have cardinality  $2n$ and therefore
$$G=xG'\cup yG'\cup \langle \, z\,  \rangle\subseteq \langle \, x \,  \rangle^G \cup \langle \, y^G \,  \rangle\cup \langle \, z \,  \rangle^G$$
as claimed.
\end{proof}

\begin{rk}\label{why M and not A} \rm
Our next goal is to prove that there are infinitely many even $g$ for which there are no
gpnf-actions of order greater than $8g$. The papers \cite{A,M} mentioned in the Introduction
are clearly related to this purpose
since they prove that there are infinitely many $g$ for which a surface of genus  $g$
has no more than $8g+8$ self-homeomorphisms.
However, from the perspective of our task,  they essentially differ.
In the first,   these $g$ are odd while in the second,   they are even.  Hence,  only the second
one concerns our case. Additionally, it is worth mentioning that the  $g$'s from  \cite{M} have
the  form $89p+1$, where $p$ is a prime satisfying a number of conditions. Therefore to achieve
our goal, it is enough to prove the following two lemmas.
\end{rk}

\begin{lem}\label{8g+8 is not gpnf}
There are no surfaces of genus $g$ allowing gpnf-actions of order $8g+8$.
\end{lem}
\begin{proof} By  \cite{A,M} and \cite{Kul} there are only two actions of order $8g+8$  on a surface
of genus $g$. They correspond  to two groups $G=G_g$ and $K=K_g$ of order $8(g+1)$ with the
following presentations
\begin{equation}\label{a-mcl}
\langle \, x,y \,:\,   x^2, y^4, (xy)^{2(g+1)}, (xy^2)^2 \,  \rangle
\end{equation}
for arbitrary $g$  and
\begin{equation}\label{kulk}
\langle \, x,y \,:\,   x^2, y^4, (xy)^{2(g+1)},   y^2(xy)y^2(xy)^g\,  \rangle
\end{equation}
if additionally  $g \, \equiv \,  3 \, (4)$. We shall show that none of these  actions  has
the gpnf property. For, observe that
$y^2$ is a central element of the group $G$ given in (\ref{a-mcl}),
$z=(xy)^2y^2  =[x,y]$ and  $z^x=z^y=z^{-1}$. Hence $G'$ is a cyclic  group of order $n=g+1$
or $n=2(g+1)$. In the first case $G_{\rm ab} = Z_2 \times Z_4$, which is not a gpnf-group,
while in the second case, $z$ is neither  conjugated to a power of $y$ nor to  a power of  $xy$,
so $G$ is not a gpnf-group.

Now the group $K$ given by (\ref{kulk}) has order $8(g+1)$ by \cite{Kul} and for the normal
closure  $N$  of $(xy)^n$ in $K$, we have
$$
K/N  =
\langle \, x,y \,:\,   x^2, y^4, (xy)^{n},   y^2(xy)y^2(xy)^{-1}\,  \rangle
$$
which is simply $G_{n/2}$ from the previous case, and hence also  $K$ is not a gpnf-group.
\end{proof}

 \begin{lem}\label{between 8g and 8g+8}
 There are no  gpnf-actions  of  groups $G$ of orders $8g+s$ on  surfaces  of genera
 $g=89p+1$ for arbitrary   $1\leq s \leq 7$ and any prime $p>84\!\cdot \!89$.
  \end{lem}
\begin{proof}
 First, we shall show that a group  $G$  admitting such action is a smooth factor
 group $\Delta/\Gamma$ for a Fuchsian triangle group $\Delta$, say with signature $(k,m,n)$.
Indeed, if this is not the case, then the orbit genus $h$ of $\Delta$ is non-zero or $r\geq 4$.
If $h\neq 0$ or $h=0, r\geq 5$  then  $|G| \leq 4(g-1)$  by the Riemann-Hurwitz  formula.
Therefore let  $h=0$ and $r=4$.
However, if at least two periods of $\Delta$ are greater than $2$, then  $|G| \leq 6(g-1)$.
Thus, it remains to consider $\Delta$ with signature $(0;2,2,2,n)$, where $G= (4n/(n-2))(g-1)$.
But for $n=3$, $|G|=12(g-1)$, while for $n>4$, $|G| \leq 8(g-1)$, completing this part of
the proof.
By the gpnf-property of $G$ any prime divisor of $|G|$ divides  one of  $k,l,m$, and without
loss of generality, we can assume that $k\leq m \leq n$.

Now, observe that for  $k\geq 4$, $|G| \leq 8(g-1)$ by the Riemann-Hurwitz formula.
 So let $k=3$. Then for $m\geq 5$, $|G| \leq (15/2)(g-1)$. For $m=4$,
 $|G|=24n(g-1)/(5n-12)$ and  thus for  $n \geq 6$,
 $|G| \leq 8(g-1)$.  For $n= 5$,
$|G| =(120/13)(g-1)$
which cannot be the case given the form of $g$, while for $n=4$, $|G|=12(g-1)$ which
is also not possible.

Therefore we can assume that $k=2$
and $3\leq m \leq n$.
Then
\begin{equation}\label{2mn}
 |G|=  \dfrac{4mn}{(m-2)(n-2) -4}(g-1)
\end{equation}

\medskip
\noindent
For $m=3$  we obtain
$$
|G|=\dfrac{12n}{n-6}(g-1)\geq 12(g-1)>8(g+1),
$$
but the assertion also follows directly from Theorem \ref{arbitrary}.

\medskip 
\noindent
For $m=4$
$$
|G|=\dfrac{4n}{n-4}(g-1)
$$
which cannot be equal to $8g+s$ since otherwise
$$
(32-4n)(g-1)= (8+s)(n-4)
$$
which is not the case for any $s$ since $4<n<8$, while $g-1>89p$.

\medskip 
\noindent
Finally, suppose that  $5 \leq m \leq n$.
Then for $n\geq 20$,
(\ref{2mn}) gives
$|G|\leq 8(g-1)$. So let $5 \leq m \leq n <20$. Then on the one hand, $(m-2)(n-2)-4$ does not
divide $4mn$. On the other hand
 $(m-2)(n-2)-4 < 320$,  which is impossible  since $(m-2)(n-2)-4 $ divides
 $356pnm$.
 \end{proof}

All of these, taken together,  give the following counterpart of the Accola-Maclachlan
result for gpnf-actions on surfaces of even genera.

\begin{thm}\label{A-M for even gpnf}
 If $G$ is the largest  gpnf-action on a topological surface of even genus $g\geq 2$, then
 $|G| \geq 8g$. Moreover, this lower bound is exact for infinitely many even $g$.
 In other words,  there are infinitely many even  $g$ for which there are  no gpnf-action
 on a surface of genus $g$ of order greater than $8g$.
\end{thm}
\begin{proof}
The first part follows from  Lemma \ref{geq 8g}.
By Lemma \ref{8g+8 is not gpnf}, there are no gpnf-actions of order $8(g+1)$.
Let now assume that $g=89p+1$, where $p>84 \cdot 89$ is a prime integer.
By Remark \ref{why M and not A}, there are no conformal actions of order bigger than $8(g+1)$ \cite{M}, and by Lemma \ref{between 8g and 8g+8}, 
there are no gpnf-actions of order $8g+s$  for $s = 1, \ldots, 7$ on a surface of genus $g$.
\end{proof}

\begin{ex}[Geometric models of gpnf-actions of order $8g$]\label{model 8g}\rm
Let $g \geq 2$, and let $S$ be the closed hyperelliptic Riemann surface of genus $g$ defined
by the  equation
$$
 w^{2}=z(z^{2g}-1).
$$
 This surface admits the following conformal automorphisms:
$$
 U(z,w)=(e^{\pi i/g}/z,-w/z^{g+1}), \; V(z,w)=(e^{\pi i/g}z,e^{\pi i/2g}w).
$$
of orders $2$ and $4g$ satisfying $UVU=V^{2g-1}$.
Additionally, these relations define, as considered in Lemma \ref{geq 8g}, the generalized
quasi-dihedral group
$$
G=\langle \, U,V \,:\,  U^{2},V^{4g}, UV^{-1}UV^{2g-1} \,\rangle
$$
of order $8g$ with the orbit space $S/G$ having signature $(0;2,4,4g)$. For $g \geq 3$ it is
known that ${\rm Aut}(S)=G$. If $g \geq 2$ is even, then this conformal action of $G$  has
the gpnf-property.
\end{ex}

\subsection{The odd genus case}
Here, for arbitrary $g$ we construct a gpnf-action of order  $4g$.  We present certain
 arguments for why $4g$  may be the value of $\mu$  for infinitely many odd $g$.  Finally, we
 provide some reasons why a proof that this is actually the case, cannot be easy.

\begin{lem}\label{8g for odd}
The  group $G$ with the presentation  {\rm (\ref{atain 8g})} from the proof of
{\rm  Lemma \ref{geq 8g}}, which was  the one giving a gpnf-action of order $8g$ for even $g$,
also has order $8g$ for odd  $g=n$, but it has  no gpnf property for  such $g$.
\end{lem}
\begin{proof}
Observe that $G/G'=\langle \, x,y,z\, : \, x^2,y^4, z^{4n},y^2z^{2n}, xyz, xyxy^{-1} \, \rangle = Z_2 \times Z_4 $
because the relations $y^2z^{2n}$ and $z^{4n}$ are superfluous, while $Z_2 \times Z_4$
is not a gpnf-group.
\end{proof}

The next proposition does not essentially approach us to finding the lower bound
for the orders of gpnf-actions on surfaces of odd genus, but it seems to be interesting
in itself in the context of the mentioned Accola-Maclachlan results. In \cite{A,M}, they proved
that  if $\mu (g)>8(g+1)$, then
$\mu (g) \geq  8(g+3)$ and also this bound is exact for infinitely many odd $g$. The next
proposition shows that also these actions have no gpnf property.

\begin{prop}\label{8(g+3) for odd}
The $(2,4,n)$-action of order $8(g+3)$ on a surface of genus $g$ corresponding to the presentation
$\langle \, x,y,z\,:\, xyz, x^2, y^4, z^n, (xz^3)^2 \, \rangle $
has no gpnf-property.
\end{prop}
\begin{proof} Clearly $y^{-1} z^3 y = z^{-3}$, and so $y^2$ commutes with $z^3$.  So if  $z^3$
has order greater than $2$, this is also the case for $y^2z^3$, Therefore, if $n$ is even,
then $z^3$ and $y^2z^3$ have the same orders. Now, clearly $y^2 z^3$ does not belong to
$\langle\, z \, \rangle $ and so it also does not belong  to any of its conjugations, which
means that  $\langle\, x \, \rangle, \langle\, y \, \rangle$ and   $\langle\, z \, \rangle $
do not cover the whole group  $G$.
\end{proof}

The above facts allowed us to realize  that to find    the lower bound for $\mu(g)$  with
odd $g$ we need to  study  gpnf-actions on surfaces of odd genus of smaller orders and we get

\begin{thm}\label{geq 4g}
Given $g\geq 2$, there is a gpnf-action of order $4g$ for arbitrary $g \geq 2$.
In particular $\mu (g) \geq 4g$ for odd $g$.
\end{thm}
\begin{proof}
Observe that $G=\langle\, x,y,z \,:\, xyz, x^4,  x^2y^2, x^2z^{n}\,\rangle $
is a $(4,4,2n)$-group
of  order $4n$ and so by the Riemann-Hurwitz formula, it acts on a surface of genus $g=n$
and we shall show that it is a gpnf-group.
First, observe that $z^{-1}=xy$, and  $z =  y^{-1} x^{-1} =y^3 x^3= y(y^2x^2)x= yx$.
So
$x z x^{-1} =   xy = z^{-1}$,
 and therefore
 $z^{-i} x z^i =
z ^{-i} (x   z^i x^{-1})x=
z^{-2i}x$.
which means that the centralizer of $x$ in $G$  coincides with $\langle \, x \, \rangle $,
and hence  $x$ has exactly $n$  conjugations, and clearly the same holds for $y$.
But $x$ and $y$  are not conjugated in $G$
since they are not conjugated in $G/< z^2>$,  which is the quaternion  group of order $8$.
Finally  $x,y \not \in  \langle \, z \, \rangle $, and so their conjugacy classes lie out
of $\langle \,  z \, \rangle $ and having $n$  elements  together with $\langle \, z \, \rangle $ they cover $G$.
\end{proof}

To close the matter of gpnf-actions, it remains to show that there are  infinitely many odd
$g$ for which there are no gpnf-actions
of order greater than $4g$ on surfaces of genus $g$. This seems to be far from
being trivial.
In short, the point is that the spectrum of action orders ranging between
$4g$ and $8g$ to be considered is much wider than the spectrum of orders greater than $8g$ of
gpnf-actions to be eliminated, due to Lemma \ref{between 8g and 8g+8} and ready-to-use result
of   Accola-Maclachlan.
However, our efforts so far have led us to venture into the following.

\begin{conject}\label{conj 4g} \rm
 There are infinitely many $g$  for which $\mu(g) =4g$, which means that there are infinitely
 many odd $g$ for which there are no gpnf-actions of order greater than $4g$ (likely for all but
 finitely many or perhaps for all).
\end{conject}

\begin{ex}[Geometric models of gpnf-actions of order $4g$]\label{model 4g}\rm
 If, in  Example \ref{model 8g},   $g$ is odd then  the automorphism $UV^{2}$ has no fixed
 points. However,  there is a subgroup
$$H=\langle V^{2},VU\rangle$$ of $G$ of index $2$
which is a gpnf-action. Observe that  $H$ is a dicyclic group of order $4g$ as considered in
Theorem  \ref{geq 4g}.
\end{ex}

\section{The asymptotic of the function $\mu=\mu(g)$}
By  Theorem \ref{A-M for even gpnf}, $\mu(g)\geq 8g$    for arbitrary even $g$ and $\mu(g)=8g$
for infinitely many of them. The results of this chapter can be thought of as a reinforcement
of this theorem by providing a full description of the asymptotics of  $\mu$. We will carry out the
discussion based on the previous Theorem \ref{arbitrary} and the following Lemma concerning
all signatures distinct than $(2,4,n)$, which, by the Riemann-Hurwitz formula,  may eventually
lead to gpnf-actions of order greater than $8g$.

\begin{lem}\label{33n & 344}
There are no gpnf $(3,4,4)$-groups, gpnf $(3,3,n)$-groups for  $n=4,5,6,8,9,10,11$ and the
only gpnf   $(3,3,7)$-group is $\langle a,b\,: \,  a^3,b^7,aba^{-1}b^5 \rangle$.
\end{lem}
\begin{proof}
To begin, observe that, by Burnside {\it pq}-theorem, all groups listed in the lemma,
except $(3,3,10)$-group, should  be solvable. The Sylow $2$-subgroup of a gpnf   $(3,3,10)$-group is
elementary abelian, and $G$ has no elements of order $6$.  Therefore, $G$ is  not simple,
otherwise, by Lemmas \ref{elementary abelian_2_groups} and \ref{elements of order 6},
$G\cong {\rm PSL}(2,p^m)$, where $p \in \{2,\, 3,\, 5\}$.
Next according to Satz II.8.5. from \cite{Huppert}, the set of orders of maximal cyclic
subgroups of the group  ${\rm PSL}(2,p^m)$ is equal to
$\{2, 2^m-1, 2^m+1 \}$ for $p=2$ and $\{p, (p^m-1)/{2}, (p^m+1)/{2}\}$
for  $p$ odd.  This set has three  elements, while the similar set for our
group $G$  consists of two elements $3,10$, due to the gpnf property. This contradiction proves
that $G$ is not simple and hence it contains a nontrivial and proper normal subgroup $H$.
Again due to the gpnf property, $H$  does not have elements of order  $3$ and so it is
solvable by the mentioned Burnside Theorem.
Next, $H$ contains an element of order $2,5$ or $10$. In the last case $G/H=Z_3$
and so $G/H$ is solvable. If $H$ contains an element of order $p\in \{2,5\}$,
then it  contains  all elements of  order $p$, so $G/H$  does not contain  an element
of order $p$. This in turn means that  $G/H$ is solvable by Burnside Theorem once more.
Summing up,  in any case,  $H$ and $G/H$ are solvable, and so  $G$ is solvable as well.

Assume now that $G$ is a gpnf   $(3,3,p)$-action for  $p\in\{5,7,11\}$.
Then $G/G'\cong Z_3$,  and $G'=Z_p^s$ by the gpnf property of the action of  $G$.
 Now  $G$  acts by conjugation on the set of cyclic subgroups of $G'$ and,
the gpnf property implies that such an action has only one orbit $G_H$,
where $H=\langle z \rangle$. However, since $G'$ is abelian
and $G/G'\cong Z_3$,  the stabilizer of $H=\langle z \rangle$ coincides with
$G$ or  it   has  index  $3$ there. Therefore, the orbit $G_H$ of $H$ has cardinality $1$ or $3$.
The latter is impossible since
$s>1$ and so $G'$ has at least $(p^2-1)/(p-1)=p+1$ cyclic subgroups while the
former means that $G'$  is cyclic as claimed. It is not central since
otherwise $G$ would have an element of order $3p$.  Hence, we have a nontrivial action
$Z_3 \to {\rm Aut}(Z_p)\cong Z_{p-1}$, and so $p=7$; the mentioned group of order $21$
is the only one that appears here.

Now let $G$ be a gpnf  $(3,3,4)$-action.
Then an elementary abelian normal subgroup $K$ of $G$ is equal to $Z_2^s$,
and so $G/K$ is a $(3,3,2)$-group. Therefore, $G/K=A_4$ since this is the only
$(3,3,2)$-group. Furthermore,  $K$ contains all subgroups of order $2$ in $G$,
and  $G$, acting  on  them by conjugation, has just one orbit $G_H$, where $H=\langle z^2 \rangle$.
But the normalizer of $H$ in $G$ has index $1,2,3$ or $6$ since it contains $z$.
In the first case $xz^2$ is an element  of order $6$ in $G$, while in the remaining
cases we have $2^s-1=2,3$ or $6$.  So $s=2$ and in particular $G'$  is a group of order $16$.
Now $G'$ cannot be abelian since then   $G'=Z_4^2$  and it has six cyclic subgroups
of order $4$. However, on the one hand, all of them are conjugated in $G$ by the gpnf property,
while on the other hand, the orbit of this action has cardinality  $3$ since $G'$
is a normal subgroup of $G$ of index $3$, which is a contradiction.
So assume that $G'$ is non-abelian.  Then it has a central (in $G'$) element of
order $2$. But since all elements of order $2$ are conjugated in $G$, all of them
are central in $G'$ and therefore  $K=Z_2^2$  is a central subgroup of $G'$. But then,
$K'\cong Z_2$ since  $G'/Z(G')$ is an abelian group of order $4$. Now $K'$,  being characteristic
in $K$, is normal in $G$ and then central in $G$, which is the final contradiction.
The above also show that  there are no gpnf $(3,3,8)$-groups since for such $G$, $G/K$,
for elementary abelian $2$-subgroup of $G$,  is a gpnf $(3,3,4)$-group, which was just ruled out.

A gpnf  $(3,3,9)$-action of a group $G$ is a  non-abelian   $3$-group and therefore it
contains a normal subgroup $H$ of index $9$. So $G/H=Z_3^2$ which is a contradiction since
$Z_3^2$ is not a union of three cyclic groups.

Next, let   $G$ be a gpnf $(3,3,6)$-action. Then $G/G'=Z_3$. If
$G'$ is abelian, then it is a union of three cyclic groups, and  due to
Lemma \ref{union of cyclics}, it is isomorphic to $Z_2^2 \times Z_3$.
However, this means that a  Sylow $2$-subgroup  $S$  of $G'$, being characteristic in $G'$,
is normal in $G$ and therefore $G/S$ is a noncyclic group of order $9$, which is a contradiction.
So  $G'$
is non-abelian and  $G'/G''$ is a  $2$-group, being a union of three subgroup of order $2$.
 Thus $G'/G'' \cong Z_2^2$, $G/G'' \cong A_4$  and  $G''$  is an elementary abelian $3$-group.
 If  $G''$
is cyclic,  then  $G'=D_3$, which means that certain  elements of $G'$ of order $2$  do not belong
to subgroups of $G$ of  order $6$ .
So $G''$ is not cyclic, say $G''  \cong Z_3^s$, and therefore it has $(3^s-1)/2$  cyclic subgroups.
However, by the gpnf property of $G$    all of them are conjugated in $G$,
while from another perspective, there are $3$ or $6$  of them since the normalizer
of an arbitrary cyclic subgroup of $G''$  has index $3$ or $6$  in $G$. A contradiction.

Now let $n=10$ and let $G$ be a gpnf  $(3,3,10)$-group.
Then $G/G'=Z_3$ and thus    $G'$
is a $(10,10,10)$-group,  by  the Lemma \ref{sing}.
So if
$G'$ is abelian, then it has order $5^s 2^t$ for $s,t \leq 2$. But then a Sylow
$5$-subgroup $P=Z_5^s$ of $G'$ is normal in $G$ and $G/P \cong A_4$. Therefore
$s=2$  since otherwise $P=Z_5$ would be central in $G$, contrary to its gpnf-property.
But then we would have a faithful action of $A_4$ on $Z_5^2$. Equivalently $A_4$ would be
a subgroup of ${\rm GL}(2,5)$ which is well known not to be the case.  For non-abelian
$G'$, $G''$ is a Sylow $p$-subgroup of $G'$ for $p=2$ or $p=5$, due to the gpnf property
of $G$. But then we get a contradiction in the same way since $G'/G''$ has the rank $2$.

To complete the proof
consider a gpnf $(3,4,4)$-group $G$. Then $G/G'$ is either $Z_2$ or $Z_4$.
But in the first case, all elements of $G'$
have orders $2$ or $3$. Therefore, a Sylow $2$-subgroup $H$ of $G'$ is elementary abelian.
The subgroup $H$ is obviously normal in $G$, $G/H\cong D_3$ and no non-identity element of $H$
is central because $G$ does not have elements of order $6$. Now $G$,  acting on $H$ by
conjugation, has at most two orbits, and since any element of $G$ of order $2$ is a square of
an element of order $4$, we see that these orbits have cardinality $3$. So all elements of
order $2$ form exactly one orbit and then $|H|=4$, otherwise $H$ consists of $6$ elements of
order $2$. This implies that $|G|=24$ and the set $G-G'$, consisting of all elements of order
$4$, is distributed among $3$ Sylow $2$-subgroups. Thus a Sylow $2$-subgroup of $G$ has to
be isomorphic to $\mathbb{Z}_4\times\mathbb{Z}_2$ because it is the unique (up to isomorphism)
noncyclic group of order $8$ with exactly $4$ elements of order $4$. The subgroup $H$ is the
intersection of all Sylow $2$-subgroups of $G$ since it is normal in $G$, and then it centralizes
all elements of order $4$. This is a contradiction because it would mean that $H$ is contained
in the center of $G$, as it is generated by elements of order $4$.

So $G_{\rm ab}= Z_4$ and $G'$ is an elementary abelian $3$-group, say of order $3^s$. Now,
$G$ acting on the set of its cyclic subgroups of $G'$, has one orbit which has $1,2$ or
$4$ elements. Therefore
$3^s-1= 2,4$ or $8$,  and thus  $s=1$ or $s=2$. But for $s=1$, $|G|=12$, which is impossible since on the
one hand, among the three non-abelian groups of  order $12$, only dicyclic $Q_6$
has abelianization $Z_4$, while on the other hand, a dicyclic group  cannot be a gpnf
$(3,4,4)$-group since it has an element of order $6$.
Also $s\neq 2$ because otherwise,  $y$
permutes cyclic subgroups of order $3$, forming a $4$-cycle. But then, $y^2$ splits into
two disjoint transpositions, which is impossible since an element of order $2$ acting as
an automorphism on $Z_3^2$ either inverts all elements and so permutes cyclic
subgroups trivially or has fixed points, and then $G$ has elements of order $6$, which gives a
contradiction again.
\end{proof}

Summing up we get as a corollary the following.

\begin{thm}\label{form of mu}
Given an even  $g \geq 2$,  we have either
\begin{itemize}
\item[(a)]
$\mu(g) = 8g$ or\\[-3mm]
\item[(b)]
$\displaystyle{\frac{8n}{n-4}}(g-1)$
for some $n$
\end{itemize}
The case $(a)$ applies for infinitely many genera. For arbitrary $n$,  the case $(b)$ applies
for at most finitely many $g$.
 \end{thm}
\begin{proof}
The first part follows, in the obvious manner,  from the conjunction of
Theorem \ref{arbitrary}, Lemma  \ref {33n & 344}     and the Riemann-Hurwitz formula.
The assertion  (a) follows from Theorem \ref{A-M for even gpnf}. Part (b)
follows from the  previously mentioned Zelmanov's solution of Restricted Burnside Problem
on  a bound $B(k,e)$ for the  order of a finite
$k$-generated group of exponent $e$ (in our case $k=2, e=4n$);
the point here is  that Zelmanov's results do not permit to apply for gpnf-actions, the
 Macbeath trick used in \cite{McB} for unlimited reproduction of  Hurwitz actions.
\end{proof}

In principle, case $(b)$ may occur for infinitely many values of $n$.
However, we conjecture that this cannot be the case.

\begin{conject}\label{mu=8g for arbitrary g} \rm
$\mu (g)= 8g$ for all but finitely many $g$ (likely for all $g$).
\end{conject}

\begin{thm}\label{derivative}
The set $\mathcal{M}_+^d$  of the accumulation points
of the set
$$\mathcal{M}_+=  \left\{  \displaystyle{\frac{\mu(g)}{g+1}} : g \;  {\rm even} \right\}$$
consists of the single number $8$.
\end{thm}
\begin{proof}
Let
$
G_0= \left\{g \, : \, \mu(g)=8g\right\} \; {\rm and} \; G_n=\left\{ g \, : \, \mu (g) = \displaystyle{\frac{8n}{n-4}}(g-1)\right\}, \; n \geq 5.
$

\medskip  \noindent
By Theorem  \ref{A-M for even gpnf},  $G_0$ is  infinite while   by (b)
of Theorem \ref{form of mu},  $G_n$ is finite for all $n \geq 5$.
Then let
$$
\mathcal{M}_0= \left\{ 8\,\frac{g}{g+1} \, : \, g \in G_0 \right\}, \;\;  \mathcal{M}_n=\left\{\left(\frac{8n}{n-4}\right)\left( \frac{g-1}{g+1}\right)\,:\, g \in G_n\right\}
$$
Clearly $(\mathcal{M}_0)^d=\{8\}$ and we have to look for the accumulation points of the union
of $\mathcal{M}_n$ for $n\geq 5$.
So  let $N_n = \max \mathcal{M}_n$.
Then, since the sets $G_i$ are disjoint,  given $\varepsilon > 0$, there is $N_\varepsilon$  so that
for $n\geq N_\varepsilon$,
$8-\varepsilon \leq N_n \leq 8+\varepsilon$.
Therefore, keeping in mind that the union
$$
\bigcup_{0<n<N_\varepsilon} \mathcal{M}_n
$$
is finite we get the assertion.
\end{proof}

\begin{rk}\rm
Finally, if Conjecture \ref{conj 4g} is true, then the set  $(\mathcal{M}_{-})^d$ of the
accumulation points of
\begin{equation}\label{mathcal M odd}
\mathcal{M}_{-} = \left\{  \displaystyle{\frac{\mu(g)}{g+1}} : g\geq 3,  \; g \;  {\rm odd} \right\}
\end{equation}
is composed of the single value $4$ which in turn gives   $  \mathcal{M}^d =\{4,8\}$.
\end{rk}

\bigskip \noindent
{\bf Acknowledgements.} The authors would like to thank Natalia Maslova  and
Alireza Abdollahi for bringing the paper \cite{VV1} of A.V. Vasiliev and   E.P. Vdovin, to our
attention.
Part of the contribution of the second author was made during the workshop
{\it Symmetries of Surfaces, Maps and Dessins} at BIRS (Banff International Research Station)
in Canada, organized by Marston Conder  in  September 2017 and during his  numerous visits
to UNED in Madrid, partially supported by Spanish grants.


\end{document}